\documentclass[11pt]{article}
\clearpage{}%

\usepackage[utf8]{inputenc}
\usepackage[T1]{fontenc}

\usepackage[verbose=true,
letterpaper,
textheight=8.4in,
textwidth=6.1in,
margin=1in,
headheight=12pt,
headsep=25pt,
footskip=30pt]{geometry}

\usepackage[toc,page,header]{appendix}
\usepackage{titletoc}

\usepackage{microtype}
\usepackage{mathtools}
\usepackage{booktabs}
\usepackage[round]{natbib}
\usepackage{graphicx}
\usepackage{amsmath,amssymb,amsfonts,amsxtra,bm}
\usepackage{amsthm}
\usepackage{xspace}
\usepackage{paralist}
\usepackage{enumerate}
\usepackage{enumitem}
\usepackage{hyperref}
\usepackage{url}
\usepackage{colortbl}
\usepackage{makecell,multirow}       %
\usepackage{nicefrac}       %
\usepackage{thmtools}  
\usepackage{xspace}
\usepackage{footnote, tablefootnote}
\usepackage{float}
\floatstyle{plaintop}
\restylefloat{table}
\usepackage{epstopdf}
\usepackage{latexsym}
\usepackage{graphicx}
\usepackage{bbm}
\usepackage{cleveref}

\titlecontents{section}[3em]{\vspace{-1pt}}%
    {\bfseries\contentslabel{2em}}%
    {}%
    {\titlerule*[0.5pc]{.}\contentspage}%
    
\titlecontents{subsection}[5em]{\vspace{-1pt}}%
    {\contentslabel{2em}}%
    {}%
    {\titlerule*[0.5pc]{.}\contentspage}%

\titlecontents{subsubsection}[5em]{\vspace{-1pt}}%
    {\footnotesize\contentslabel{3em}}%
    {}%
    {\titlerule*[0.5pc]{.}\contentspage}%

\numberwithin{equation}{section}
\clearpage{}%

\usepackage{ss}
\usepackage{tikz}
\clearpage{}%
\newcommand{\Xc}{\mathcal{X}}

\newcommand{\Dc}{\mathcal{D}}
\newcommand{\Lc}{\mathcal{L}}
\newcommand{\gr}{\nabla}

\newcommand{\reals}{\mathbb{R}}
\newcommand{\half}{\tfrac12}

\newcommand{\ip}[2]{\langle {#1},\, {#2} \rangle}
\newcommand{\norm}[1]{\| {#1} \|}

\newcommand{\dc}{DC\xspace}
\newcommand{\ind}{\chi}

\DeclareMathOperator*{\argmin}{argmin}
\DeclareMathOperator{\prox}{prox}

\DeclareMathOperator*{\subjto}{s.t.}
\DeclareMathOperator*{\argmax}{argmax}

\newtheorem{theorem}{Theorem}[section]
\newtheorem{lemma}[theorem]{Lemma}
\newtheorem{prop}[theorem]{Proposition}

\newtheorem{corollary}[theorem]{Corollary}

\theoremstyle{definition}

\newtheorem{example}[theorem]{Example}

\numberwithin{equation}{section}

\clearpage{}%

\title{CCCP is Frank-Wolfe in disguise}

\begin{document}
\author{\name{Alp Yurtsever} \email{alp.yurtsever@umu.se}\\
   \addr{Umeå University}\\[2pt]
   \name{Suvrit Sra} \email{suvrit@mit.edu}\\
   \addr{Massachusetts Institute of Technology}
}

\maketitle

\begin{abstract}
This paper uncovers a simple but rather surprising connection: it shows that the well-known convex-concave procedure (CCCP) and its generalization to constrained problems are both special cases of the Frank-Wolfe (FW) method. This connection not only provides insight of deep (in our opinion) pedagogical value, but also transfers the recently discovered convergence theory of nonconvex Frank-Wolfe methods immediately to CCCP, closing a long-standing gap in its \emph{non-asymptotic} convergence theory. We hope the viewpoint uncovered by this paper spurs the transfer of other advances made for FW to both CCCP and its generalizations.
\end{abstract}

\section{Introduction}
\lettrine[lines=3]{\color{BrickRed}W}e study non-convex \emph{difference of convex (\dc)} optimization problems of the form:
\begin{equation}
  \label{eq:1}
  \min_x\quad f(x)-g(x),\quad x \in \Dc,
\end{equation}
where both $f$ and $g$ are smooth convex functions and $\Dc$ is a constraint set that might itself be non-convex (we will specify its structure later). Formulation~\eqref{eq:1} is a smooth \dc program, for which a powerful local method is Convex-Concave Procedure (CCCP)~\citep{yuille2003concave} that reduces~\eqref{eq:1} to a sequence of convex problems. CCCP is not only algorithmically appealing, it is also widely applicable and numerous algorithms can be viewed as special cases of CCCP, most notably the famous EM algorithm---see \citep[\S4][]{yuille2003concave}.

But despite its wide applicability (see \S\ref{sec:back-cccp}) and elegance, convergence theory for CCCP is quite limited. Its \emph{asymptotic} convergence can be obtained by viewing it as a specialization of the \dc algorithm (DCA)~\citep{le2018dc}, or via a more refined analysis designed for CCCP in~\citep{lanckriet2009convergence}. A long-standing gap in the literature on CCCP has been to obtain \emph{non-asymptotic} convergence theory that establishes convergence to an $\varepsilon$-stationary point in $O(\text{poly}(1/\varepsilon))$ or fewer iterations. %

The starting point of this paper is a discovery that is elementary yet surprising: \emph{CCCP is a special case of the Frank-Wolfe (FW) method}! This discovery not only provides insight of deep (in our opinion) pedagogical value, but it also transfers the recently discovered convergence theory of nonconvex FW to CCCP, which closes the abovementioned gap in CCCP's non-asymptotic convergence theory.

\subsection{Main contributions}
\vspace*{-3pt}
In light of this short background, we summarize now the key contributions of this paper.
\begin{list}{{\color{darkred!90}\tiny$\blacksquare$}}{\leftmargin=2em}\vspace*{-5pt}
  \setlength{\itemsep}{2pt}
\item We recognize the basic (unconstrained) CCCP method to be a special case of Frank-Wolfe. By the same argument, even its generalization to convex constrained instances of~\eqref{eq:1} is a special case of FW. This realization allows us to immediately transfer recently discovered non-asymptotic convergence theory of FW to CCCP and convex constrained CCCP.
\item Building on the above connection, we subsequently propose a new variant of FW (called FW+) that applies to the most general form of CCCP (called CCCP+), where the constraint set $\Dc$ in \eqref{eq:1} is specified via several \dc constraints. The variant FW+ not only allows us to syntactically view CCCP+ as a special case, but more importantly allows us to transfer non-asymptotic convergence guarantees that we prove for FW+ directly to CCCP+.
\end{list}
While the FW+ variant and its analysis are slightly more subtle than basic FW, the connection between CCCP and FW uncovered by our work is remarkably simple, which makes it all the more surprising that it was missed in the vast body of previous work on FW and CCCP. The introduction of FW+ itself is an immediate consequence of the connection uncovered. It is also quite plausible that the connections described in this paper will have various further consequences, including development of new variants of CCCP and other related methods. %
For a more concrete discussion of implications, see the discussion in Section~\ref{sec:disc}.

Finally, we note a few additional connections of the view uncovered in this paper in the appendix. In particular, we show that various other methods such as the proximal point method, mirror descent and mirror prox can also be seen as special instances of the FW method. A deeper exploration of these connections is left as future work. 

\vspace*{-2pt}
\subsection{Related work}
\vspace*{-4pt}
\textbf{CCCP.} The CCCP method of~\citet{yuille2003concave} is a special case of the more general \dc algorithm (DCA)~\citep{tao1997convex,le2018dc}.  CCCP often provides a strong baseline method for tackling non-convex problems where a \dc structure is already known or relatively easy to obtain--- see~\citep{lipp2016variations} for some recent variations on the CCCP method. As already noted, asymptotic convergence of CCCP follows from the more general convergence theory of DCA~\citep{tao1997convex}, while a simplified, more direct convergence analysis was obtained in~\citep{lanckriet2009convergence}, who analyzed convergence of both function values and iterates; moreover, these convergence results apply to the most general CCCP+ formulation (\dc objective with \dc constraints). A more recent work~\citep{khamaru2018convergence} considers gradient and subgradient based alternatives to DCA and CCCP, establishing their corresponding non-asymptotic convergence; however, its focus is on alternative and its guarantees do not transfer to CCCP, and thus \emph{a fortiori} not to CCCP+. Both CCCP and DCA have received a large number of applications, in a variety of areas including statistics, machine learning, signal processing, operations research, and many more---we refer the reader to the extensive list documented in the survey~\citep{le2018dc} and to examples in~\citep{yuille2003concave}.

After this work was completed, we discovered an important item of related work~\citep{abbaszadehpeivasti2021rate}. The authors therein propose non-asymptotic convergence rate guarantees for unconstrained DC optimization. They present two main results: (i) Under the assumption that $f$ and $g$ are convex and at least one of them is smooth, they show $\min_{\tau \in [k]} \norm{v_\tau - u_\tau} \leq \mathcal{O}(1/\sqrt{k})$ where $v_\tau \in \partial f(x_\tau)$ and $u_\tau \in \partial g(x_\tau)$; %
 (ii) For the non-smooth setting, they show $\min_{\tau \in [k]} f(x_\tau) - f(x_{\tau+1}) - \ip{u_\tau}{x_\tau - x_{\tau+1}} \leq \mathcal{O}(1/k)$ where $u_\tau \in \partial g(x_\tau)$. 
This second result matches our guarantees in \Cref{cor:basic-cccp}. As we will also note below, \Cref{cor:basic-cccp} does not require any assumption on the smoothness of functions, and its guarantees also hold with subgradients. But these guarantees (and therefore those of the above cited paper) have an important \emph{caveat}: even if the ``gap'' $f(x_k) - f(x_{k+1}) - \ip{u_k}{x_k - x_{k+1}} = 0$, it does not imply stationarity when $g$ is non-differentiable. Thus, a more careful study of the non-smooth case remains open. 

Our analysis is fundamentally different from the methods used in \citep{abbaszadehpeivasti2021rate}. Their proof is based on the performance estimation technique of \citet{drori2014performance}, which is a computer assisted methodology for estimating the worst-case complexity of an algorithm by solving semidefinite programs. They conjecture the convergence bounds via systematic numerical experimentation and then verify these bounds analytically. Hence, the proof provides limited insight and does not extend to the more general constrained settings studied in our paper. In contrast, our analysis is based on a simple yet deep connection between CCCP and FW. This deeper recognition allows us not only to extend our guarantees to  constrained CCCP (called CCCP+ in the main paper), but also leads to the discovery of FW+.

\textbf{FW.} %
\citet{fw_original} proposed the FW algorithm originally for minimizing a convex quadratic function over a polytope and proved $\mathcal{O}(k^{-1})$ convergence rate in this setting.  \citet{levitin1966constrained} extended the analysis and proved the same rate for minimization of an arbitrary smooth convex function over a generic convex compact set, and \citet{canon1968tight} showed that this rate is optimal. The analysis is extended for a broader H\"older-smooth function class by \citet{nesterov2018complexity}. \citet{jaggi2013revisiting} presented an affine invariant analysis for the FW algorithm by replacing smoothness with the bounded curvature assumption, see \eqref{eqn:curvature-constant}. The analysis of FW for non-convex functions appeared in \citep{lacoste16}---the method finds a stationary point as the gap function converges to zero with $\smash{\mathcal{O}(k^{-1/2})}$ rate, see \eqref{eqn:fw-gap} for the gap function. 

Design variants of FW with away, pairwise, or in-face steps and line-search or momentum strategies; and extensions of FW for stochastic gradients or block-coordinate updates are extensively studied in the literature. We refer to \citep{kerdreux2020accelerating} for an excellent overview on these recent developments. To our knowledge, the variant FW+ introduced in this paper is new.

\section{Background on Frank-Wolfe and CCCP}
\subsection{Frank-Wolfe Algorithm}
\label{sec:back-fw}
We begin by recalling the basic Frank-Wolfe method (FW) and some of its key properties. We focus in particular on the nonconvex case, as it is key to the rest of our theory. Consider therefore the prototypical constrained optimization template: 
\begin{equation}
\label{eq:concave-min}
    \min_{\omega} \ \ \phi(\omega)\quad\subjto\quad\omega \in \Dc,
\end{equation}
where $\Dc$ is a closed and convex set in $\reals^n$ and $\phi$ is a continuously differentiable (potentially nonconvex) function over the domain $\Dc$. FW applied to~\eqref{eq:concave-min} performs the following steps for $k=1,\ldots$
\begin{equation}\label{eqn:FW}\tag{FW}
\begin{aligned}
    & \omega_k^* \in \argmin_{\omega \in \Dc} ~~ \phi(\omega_k) + \ip{\gr \phi(\omega_k)}{\omega - \omega_k},\\
    & \omega_{k+1} = (1-\eta_k) \omega_k + \eta_k \omega_k^*,
\end{aligned}
\end{equation}
where $\eta_k \in [0,1]$ is a step-size. There are multiple design choices for the step-size. A common choice is the greedy step-size 
\begin{equation}\label{eqn:FW-stepsize}
    \eta_k = \argmin_{\eta \in [0,1]} ~~~ \phi\big ((1-\eta) \omega_k + \eta \omega_k^* \big).
\end{equation}
The standard convergence theory of FW crucially depends on the boundedness of \emph{curvature constant} of $\phi$ over $\Dc$, defined as follows:
\begin{equation}\label{eqn:curvature-constant}
    C_\phi := \sup_
    {\substack{\omega,\hat{\omega} \in \Dc;~ \eta \in [0,1] \\[0.15em] \bar{\omega} = (1-\eta)\omega + \eta \hat{\omega}}} 
    ~~~~ \frac{2}{\eta^2}\Big( \phi(\bar{\omega}) - \phi(\omega) - \ip{\gr \phi(\omega)}{\bar{\omega} - \omega} \Big).
\end{equation}
The bounded curvature assumption is closely related to the smoothness of $\phi$. If $\gr \phi$ is $L$-Lipschitz continuous on $\Dc$, then $C_\phi \leq L D^2$, where $D$ denotes the diameter of $\Dc$. Moreover, it is easy to see that $C_\phi \leq 0$ if $\phi$ is concave. This becomes useful in our analysis later in this paper, and allows us to derive guarantees without any Lipschitz smoothness assumption on the functions involved.

If $\phi$ is convex and $C_\phi$ is bounded, then the sequence $\{\omega_k\}_{k\ge 0}$ generated by \eqref{eqn:FW} converges to a solution in objective value (see e.g.,~\citep[Theorem~1]{jaggi2013revisiting}); where the convergence rate is
\begin{equation*}
    \phi(\omega_k) - \phi^\star \leq \frac{2C_\phi}{k+1}, \quad \text{where} \quad  \phi^\star := \min_{\omega \in \Dc} \phi(\omega).
\end{equation*}
When $\phi$ is non-convex, \eqref{eqn:FW} can be shown to find a stationary point~\citep[Theorem~1]{lacoste16}, with the following guarantee: there exists an index $\tau \in [k]$, such that %
\begin{equation*}
  \ip{\gr \phi(\omega_\tau)}{\omega_\tau - \omega} \leq \frac{\max\{C_\phi, \ 2(\phi(\omega_1) - \phi^\star)\}}{\sqrt{k}} \qquad \forall \omega \in \Dc. 
\end{equation*}
This bound is often presented in terms of the so-called FW-gap:
\begin{equation}
\label{eqn:fw-gap}
    \mathrm{gap}(\omega_{\tau}) := \max_{\omega \in \Dc} ~~ \ip{\gr \phi(\omega_\tau)}{\omega_\tau - \omega}
    = \ip{\gr \phi(\omega_\tau)}{\omega_\tau - \omega_\tau^*}.
\end{equation}
This quantity naturally arises in the FW algorithm and its analysis, and it  provides a reliable stopping criterion. 

In our analysis, we particularly need to apply FW to concave objectives. In this case, \eqref{eqn:FW} exhibits better convergence rates than general non-convex costs. More precisely, we note the following result.
\begin{lemma}
\label{lem:fw-concave}
Suppose that the sequence $\{\omega_k\}$ is generated by the FW method applied to Problem~\eqref{eq:concave-min} with a concave objective. Then, there exists $\tau \in [k]$, such that
\begin{equation*}
    \ip{\gr \phi(\omega_\tau)}{\omega_\tau - \omega} \leq \frac{\phi(\omega_1) - \phi^\star}{k} \qquad \forall \omega \in \Dc. 
\end{equation*} 
\end{lemma}

\begin{proof}
First, observe that $\eta_k = 1$ in \eqref{eqn:FW-stepsize} when $\phi$ is concave. Then, \eqref{eqn:FW} becomes
\begin{equation*}
    \omega_{k+1} \in \argmin_{\omega} ~~ \phi(\omega_k) + \ip{\gr \phi(\omega_k)}{\omega - \omega_k} \quad \subjto \quad \omega \in \Dc. 
\end{equation*}
By concavity of $\phi$, we have 
\begin{equation}\label{eq:proof-concave-fw-s1}
    \max_{\omega \in \Dc} \ \ 
    \ip{\gr \phi(\omega_k)}{\omega_k - \omega}
    = \ip{\gr \phi(\omega_k)}{\omega_k - \omega_{k+1}} 
    \leq \phi(\omega_k) - \phi(\omega_{k+1}).
\end{equation}
Then, averaging \eqref{eq:proof-concave-fw-s1} over $k$ gives
\begin{equation}\label{eq:proof-concave-final}
    \frac{1}{k} \sum_{i=1}^k \max_{\omega \in \Dc} \ \ 
    \ip{\gr \phi(\omega_i)}{\omega_i - \omega} 
    \leq \frac{1}{k} (\phi(\omega_1) - \phi_\star).
\end{equation}
We conclude the proof by noting that the minimum over $k$ is less than or equal to the average.
\end{proof}

\subsection{The Convex-Concave Procedure (CCCP)}
\label{sec:back-cccp}
The starting point for the Convex-Concave Procedure (CCCP) of~\citet{yuille2003concave} is the unconstrained \dc program
\begin{equation}
  \label{eq:15}
  \min_x\quad f(x)-g(x),
\end{equation}
where both $f$ and $g$ are $C^1$ (once continuously differentiable) convex functions; CCCP can also be viewed as a special case of the more general DCA algorithm~\citep{tao1997convex} that does not require differentiability. The key idea behind the success of CCCP is to use the convexity of $g(x)$ to linearize it and obtain the following global upper bound on $f(x)-g(x)$:
\begin{equation}
  \label{eq:16}
  Q(x;y) := f(x) - g(y) - \ip{\gr g(y)}{x-y}.
\end{equation}
At each iteration, CCCP then updates its guess by solving the convex problem
\begin{equation}
  \label{eq:17}
  x_{k+1} \in\quad\argmin_{x}\ Q(x; x_k).
\end{equation}
CCCP is thus a specific MM algorithm~\citep{hunter2004tutorial}, and owing to the update~\eqref{eq:17}, it generates a monotonically decreasing sequence $\{f(x_k)-g(x_k)\}_{k\ge 0}$ of objective values. When $f$ is differentiable, update~\eqref{eq:17} is tantamount to the implicit update $\gr f(x_{k+1})=\gr g(x_k)$. Note that the subproblem~\eqref{eq:17} is a convex optimization problem that \emph{does not} approximate the function $f(x)$, and thus likely offers a tighter approximate model to the \dc cost~\eqref{eq:15}, an aspect that might help explain its strong empirical performance~\citep{lipp2016variations}.

The original CCCP paper~\citep{yuille2003concave} presented various applications that could be viewed through the CCCP lens, of which the EM algorithm~\citep{dempster1977maximum} and Sinkhorn's method for matrix scaling~\citep{sinkhorn1967diagonal} are two important special cases. A formal study of CCCP's asymptotic convergence was undertaken in~\citep{lanckriet2009convergence}, who studied all three variants: unconstrained CCCP, convex constrained, as well \dc constrained CCCP. However, even for the simplest unconstrained version, a non-asymptotic analysis of its convergence rate (as a function of iteration count $k$) has been long open, except for the preprint~\citep{abbaszadehpeivasti2021rate} that has appeared parallel to this work and addresses the unconstrained setting.

\section{Unconstrained CCCP and convex constrained CCCP via FW}
\label{sec:basic-cccp}
As noted above, for the unconstrained \dc program
\begin{equation}\label{eqn:basic-dc}
    \min_{x} \quad f(x) - g(x),
\end{equation}
CCCP iteratively performs the following implicit update:
\begin{equation}\label{eqn:basic-cccp}\tag{CCCP}
    \gr f(x_{k+1}) = \gr g(x_k).
\end{equation}

Our first result shows that CCCP can be viewed as a special case of the FW method. 

\begin{prop}
\label{thm:basic-cccp}
CCCP is equivalent to FW applied to the following reformulation of Problem~\eqref{eqn:basic-dc}:
\begin{equation}\label{eqn:basic-dc-epi}
    \min_{x,t} \quad t - g(x) \quad \subjto \quad f(x) \leq t.
\end{equation}
\end{prop}

\begin{proof}
If we apply FW to \eqref{eqn:basic-dc-epi}, the linear minimization step is
\begin{equation}\label{eqn:basic-dc-epi-lmo}
    (x^*_k, \ t^*_k) \quad \in \quad \arg\min_{(x,t)} \quad t - \ip{\gr g(x_k)}{x} \quad \subjto \quad f(x) \leq t.
\end{equation}
KKT conditions are necessary and sufficient for optimality since \eqref{eqn:basic-dc-epi-lmo} is a convex optimization problem. The Lagrangian of this problem is 
\begin{equation*}
    \Lc (x,t,\lambda) = t - \ip{\gr g(x_k)}{x} + \lambda (f(x) - t).
\end{equation*}
Lagrangian stationarity condition implies
\begin{equation} 
\label{eq:cccp-subprob}
\left. \begin{aligned}
\gr_x \Lc (x,t,\lambda^*) \big|_{x = x_k^*,\, t = t_k^*} & = -\gr g(x_k) + \lambda^* \gr f(x_k^*) = 0 \\
\gr_t \Lc (x,t,\lambda^*) \big|_{x = x_k^*,\, t = t_k^*}  & = 1 - \lambda^* = 0
\end{aligned} ~~ \right \} ~~
\gr f(x_k^*) = \gr g(x_k).
\end{equation}
We also get $t_k^* = f(x_k^*)$ by the complementary slackness condition. 

Finally, FW updates its estimate via 
\begin{equation*}
\begin{aligned}
    x_{k+1} = (1-\eta_k) x_k + \eta_k x_k^*, \quad t_{k+1} = (1-\eta_k) t_k + \eta_k t_k^*,
\end{aligned}
\end{equation*}
where $\eta_k \in [0,1]$ is chosen to minimize the objective function in \eqref{eqn:basic-dc-epi}. 
Since this function is concave, $\eta_k = 1$, and we conclude that the linear system in \eqref{eq:cccp-subprob} is exactly the CCCP subproblem.
\end{proof}

\Cref{thm:basic-cccp} recognizes a simple derivation of the CCCP update -- reformulate the unconstrained \dc program as a concave minimization over a convex set and apply FW. 
This recognition transfers the convergence theory of FW to CCCP as we present in the next corollary.

\begin{corollary}\label{cor:basic-cccp}
Suppose that the sequence $\{x_k\}$ is generated by \eqref{eqn:basic-cccp} applied to Problem~\eqref{eqn:basic-dc}. Then, there exists $\tau \in [k]$, such that
\begin{equation*}
    f(x_\tau) - f(x) - \ip{\gr g(x_\tau)}{x_\tau - x} ~~ \leq ~~ \frac{1}{k}\big(f(x_1) - g(x_1) - (f^* - g^*)\big) \quad \text{for all $x$.}
\end{equation*}
\end{corollary}

\begin{proof}
By \Cref{thm:basic-cccp}, CCCP is equivalent to FW applied to \eqref{eqn:basic-dc-epi}. Observe that \eqref{eqn:basic-dc-epi} is an instance of \eqref{eq:concave-min} with $\omega = (x,t)$,~~$\phi(\omega) = t - g(x)$, and $\Dc = \{(x,t): f(x) \leq t\}$. Since $\phi$ is concave, we can use \Cref{lem:fw-concave}. Hence, there exists $\tau \in [k]$, such that
\begin{equation*}
    t_\tau - t - \ip{\gr g(x_k)}{x_k - x} \leq \frac{(t_1 - g(x_1)) - (t^\star -  g^\star)}{k} \qquad \forall (x,t) : f(x) \leq t. 
\end{equation*} 
$t^\star = f^\star$ by definition, $t_1 = f(x_1)$ by initialization, and $t_\tau = f(x_\tau)$ by complementary slackness, see the proof of \Cref{thm:basic-cccp}. 
\end{proof}

\Cref{cor:basic-cccp} establishes non-asymptotic convergence guarantees of CCCP to a stationary point of the unconstrained \dc program \eqref{eqn:basic-dc}. Remarkably, this result does not require any assumption on the Lipschitz-smoothness of the functions.

\subsection{Convex constrained CCCP via FW}
Here the problem of interest is:
\begin{equation}
  \label{eq:10}
  \min_x\quad f(x) - g(x),\quad\subjto\quad x \in \Dc,
\end{equation}
for a closed convex set $\Dc$. We can equivalently write this problem as
\begin{equation}
  \label{eq:11}
  \min_x\quad f(x) + \ind_{\Dc}(x) - g(x),
\end{equation}
where $\ind_{\Dc}(\cdot)$ denotes the indicator function of $\Dc$. Since $h := f+\ind_{\Dc}$ is convex, we can apply the same idea as for the unconstrained case by looking at the equivalent problem
\begin{equation}
  \label{eq:12}
  \min_{x,t}\quad t - g(x),\quad\subjto\quad h(x) \le t.
\end{equation}
If we apply FW to~\eqref{eq:12}, the linear minimization step involves solving
\begin{equation}
  \label{eq:13}
  \min_{x,t}\ t-\ip{\gr g(x_k)}{x},\quad\subjto\quad f(x) + \ind_{\Dc}(x) \le t,
\end{equation}
which is clearly equivalent to solving
\begin{equation}
  \label{eq:14}
  \min_x\quad f(x) - \ip{\gr g(x_k)}{x},\quad\subjto\quad x \in \Dc.
\end{equation}
Subproblem~\eqref{eq:14} is precisely the update of convex constrained CCCP.

\section{CCCP with difference of convex constraints via FW}
\label{sec:cccp-constr}

At this point, it is natural to wonder whether the most general version of CCCP also admits a natural FW interpretation? Consider therefore the \dc programming problem with \dc constraints:
\begin{equation}
  \label{eq:2}
    \min_x \quad f_0(x)-g_0(x) \quad \subjto \quad f_i(x)-g_i(x) \le 0,\ \ i=1,\ldots,m.
\end{equation}

The \emph{advanced} CCCP for Problem \eqref{eq:2} is introduced in \citep[Theorem~1]{smola2005kernel}. 
This method updates its estimation to a solution of the following subproblem:
\begin{equation} \label{eqn:cccp+}\tag{CCCP+}
\begin{aligned}
    & \min_x & & f_0(x) - g_0(x_k) - \ip{\gr g_0(x_k)}{x - x_k} \\ 
    & ~\subjto & & f_i(x) - g_i(x_k) - \ip{\gr g_i(x_k)}{x - x_k} \le 0 ,\qquad i=1,\ldots,m.
\end{aligned}
\end{equation}

There is a major challenge to make a natural connection between CCCP+ and FW: 
The feasible set of \eqref{eq:2} is non-convex, and the feasible sets for \eqref{eq:2} and \eqref{eqn:cccp+} are not the same. 
As a result, it seems impossible to write CCCP+ as a special instance of the standard FW method. 
However, searching for a connection between CCCP+ and FW leads us to the discovery of a more general form of the FW method itself that we call FW+. This method is obtained by linearizing constraints along with the objective function, and it recovers CCCP+ as a special case while remaining amenable to a transparent FW-style convergence analysis.

\subsection{Extended FW algorithm} 
\label{sec:FW+}

Here is the problem template:
\begin{equation}
\label{eqn:FW+-template}
    \min_{\omega \in \Dc} \ \ \phi(\omega)\quad\subjto\quad \psi_i(\omega)\leq 0, \quad i = 1,\ldots,m, 
\end{equation}
where $\Dc$ is a closed and convex set in $\reals^n$ and $\phi$ and $\psi_i$ are continuously differentiable functions with bounded curvature. 

The FW+ steps for Problem~\eqref{eqn:FW+-template} are as follows:
\begin{equation}\label{eqn:FW+}\tag{FW+}
\begin{aligned}
    & \begin{aligned} 
    \omega_k^* \in & \argmin_{\omega \in \Dc} ~~ \phi(\omega_k) + \ip{\gr \phi(\omega_k)}{\omega - \omega_k}, \\ 
    & \quad\subjto\quad \psi_i(\omega_k) + \ip{\gr \psi_i(\omega_k)}{\omega-\omega_k} \leq 0, \quad i=1,\ldots,m \end{aligned}
    \\[0.75em]
    & \omega_{k+1} = (1-\eta_k) \omega_k + \eta_k \omega_k^*.
\end{aligned}
\end{equation}
Here, we focus on a setting where both $\phi$ and $\psi$ are concave. In this setting, we can choose $\eta_k = 1$. We defer a more general discussion on FW+ for other settings to an extension of this paper.

\begin{theorem}
\label{thm:fw+}
Suppose that the sequence $\{\omega_k\}$ is generated by the FW+ algorithm for Problem~\eqref{eqn:FW+-template} with concave functions $\phi$ and $\psi$. Then, there exists an index $\tau \in [k]$, at which $\psi_i(\omega_\tau) \leq 0$ for $i = 1\ldots,m$, and 
\begin{equation}\label{eqn:thm-fw+}
    \ip{\gr \phi(\omega_\tau)}{\omega_\tau - \omega} \leq \frac{\phi(\omega_1) - \phi^\star}{k}, \quad 
    \forall \omega \in \bigcap_{i=1}^m \{\omega \in \Dc : \psi_i(\omega_\tau) + \ip{\gr \psi_i(\omega_\tau)}{\omega - \omega_\tau} \leq 0 \}.
\end{equation}
\end{theorem}

\begin{proof}
First, we focus on the constraint function $\psi$. By concavity, for all $i$,
\begin{equation*}
\begin{aligned}
    \psi_i(\omega_{k+1}) 
    \leq \psi_i(\omega_k) + \ip{\gr \psi_i(\omega_k)}{\omega_{k+1} - \omega_k} 
    \leq 0,
\end{aligned}
\end{equation*}
where the last inequality follows from the constraint in the linear minimization step of FW+. 

Similarly, by concavity of $\phi$, we have
\begin{equation*}
    \ip{\gr \phi(\omega_k)}{\omega_k - \omega_{k+1}} \leq \phi(\omega_k) - \phi(\omega_{k+1}).
\end{equation*}
Averaging this inequality over $k$ gives 
\begin{equation*}
    \frac{1}{k} \sum_{i=1}^k \ip{\gr \phi(\omega_i)}{\omega_i - \omega_{i+1}} \leq \frac{\phi(\omega_1) - \phi(\omega_{k+1})}{k} \leq \frac{\phi(\omega_1) - \phi^\star}{k},
\end{equation*}
where the second inequality holds since $\omega_{k+1}$ is in a restriction of the feasible set of \eqref{eqn:FW+-template}.
Then, there exists an index $\tau \in [k]$, at which
\begin{equation*}
    \ip{\gr \phi(\omega_\tau)}{\omega_\tau - \omega_{\tau+1}} \leq \frac{\phi(\omega_1) - \phi^\star}{k}.
\end{equation*}
This inequality leads to \eqref{eqn:thm-fw+} by definition of $\omega_{\tau+1} = \omega_{\tau}^*$ in \eqref{eqn:FW+}.
\end{proof}

Before moving on, we characterize stationary points of Problem~\eqref{eqn:FW+-template} in the next lemma. This gives a clear interpretation of the bounds in \eqref{thm:fw+} as a perturbation of the stationarity condition. 

\begin{lemma}\label{lem:stationary}
Consider Problem~\eqref{eqn:FW+-template} and assume $\phi$ and $\psi$ are concave. Then, $\omega^* \in \reals^n$ is a feasible stationary point of Problem~\eqref{eqn:FW+-template} if  
\begin{gather}
\omega^* \in \Dc \quad \text{and} \quad \psi_i(\omega^*) \leq 0, \quad \forall i = 1,\ldots,m, \tag{feasibility} \\
\ip{\gr \phi(\omega^*)}{\omega - \omega^*} \geq 0, \quad  \forall \omega \in \bigcap_{i=1}^m \{ \omega \in \Dc : \psi_i'(\omega,\omega^*) \leq 0\}, \tag{stationarity}
\end{gather}
where $\psi_i'(\omega,\omega^*) := \psi_i(\omega^*) + \ip{\gr \psi_i(\omega^*)}{\omega - \omega^*}$.
\end{lemma}

\begin{proof}
A feasible point $\omega^*$ is called stationary if there is no feasible descent direction at $\omega^*$. Let's note that $\psi_i(\omega) \leq \psi_i'(\omega,\omega^*)$ due to concavity, therefore the set considered for the stationarity condition in \Cref{lem:stationary} is a restriction of the feasible set. The lemma implies that this restriction contains all feasible directions at $\omega^*$. Here is a simple proof. 

First, suppose that that the conditions in \Cref{lem:stationary} are satisfied at $\omega^*$ and the inequality constraint is active, \textit{i.e.,} $\psi_i(\omega^*) = 0$. Then, $\psi_i'$ is also active since $\psi_i'(\omega^*,\omega^*) = \psi_i(\omega^*) = 0$. 
Assume that there exists a direction $d \in \reals^n$ such that $\bar{\omega} := \omega^* + \alpha d$ is feasible for small enough $\alpha > 0$, and that $d$ is a descent direction, $\ip{\gr \phi(\omega^*)}{d} < 0$. 
By the stationarity condition in \Cref{lem:stationary}, this implies either $\bar{\omega} \notin \Dc$ or $\psi_i'(\bar{\omega},\omega^*) > 0$. 
The former clearly contradicts our feasibility assumption. Suppose the latter holds, then we get $\ip{\gr \psi_i(\omega^*)}{d} > 0$, meaning that $d$ is an ascent direction for $\psi_i$ at $\omega^*$. 
This also contradicts our feasibility assumption -- constraint is active and $d$ is an ascent direction, hence $d$ is directed outwards from the feasible region. 
By contradiction, we conclude that there is no feasible descent direction at $\omega^*$. 

Next, suppose the constraint is inactive. This means $\psi'$ is also inactive since $\psi_i'(\omega^*,\omega^*) = \psi_i(\omega^*) < 0$. Then, we can place an open ball $\mathcal{B}(\omega^*)$ of infinitesimal radius centered at $\omega^*$ such that $\psi_i'(\omega,\omega^*) \leq 0$ for all points in $\mathcal{B}(\omega^*)$. By the stationarity condition in \Cref{lem:stationary}, this inequality then implies $\ip{\gr \phi (\omega^*)}{\omega - \omega^*} \geq 0$ for all $\omega \in \Dc \cap \mathcal{B}(\omega^*)$ and proves that there is no feasible descent direction at $\omega^*$. 
\end{proof}

\Cref{thm:fw+} and \Cref{lem:stationary} together imply that FW+ provably finds a stationary point of Problem~\eqref{eqn:FW+-template} under the concavity assumptions with non-asymptotic convergence guarantees.
We complement this result in \Cref{sec:supp-fw+} by deriving convergence guarantees of FW+ when $\phi$ and $\psi$ are convex.

\subsection{CCCP+ via FW+}

We are now ready to build a connection between CCCP+ and FW+ by following similar ideas as in \Cref{sec:basic-cccp}.

\begin{prop}
\label{fact:cccp+}
CCCP+ is equivalent to the FW+ method applied to the following epigraph reformulation of Problem~\eqref{eq:2}:
\begin{equation}
\label{eqn:diff-constraints-reform}
\begin{aligned}
    \min_{x,t_0,\ldots,t^m} \quad t_0-g_0(x) \quad \subjto \quad & f_i(x) \leq t_i, & & i=0,1,\ldots,m \\
    & g_i(x) \geq t_i, & & i=1,\ldots,m.
\end{aligned}
\end{equation}
\end{prop}

\begin{proof}
Problem~\eqref{eqn:diff-constraints-reform} is a special case of \eqref{eqn:FW+-template} with
\begin{equation*}
    \omega = (x, t_0, \ldots, t_m), \quad \phi(\omega) = t_0 - g_0(x), \quad \psi_i = t_i - g_i(x), \quad \text{and} \quad \Dc = \{ \omega : f_i(x) \leq t_i \}.
\end{equation*}

We consider FW+ for \eqref{eqn:diff-constraints-reform}. Since $\phi$ and $\psi_i$ are concave, we can use $\eta_k=1$ and the \eqref{eqn:FW+} procedure accounts to solving
\begin{equation*}
\begin{aligned}
    & \min_{x,t_0,\ldots,t_m} & & t_0 - g_0(x_k) - \ip{\gr g(x_k)}{x - x_k} \\ 
    &  \quad \, \subjto & & f_i(x) \leq t_i, & & i = 0,1,\ldots,m \\[0.5em]
    & & & t_i - g_i(x_k) - \ip{\gr g_i(x_k)}{x - x_k} \leq 0, & & i = 1,2,\ldots,m,
\end{aligned}
\end{equation*}
which is clearly equivalent to the \eqref{eqn:cccp+} subproblem.
\end{proof}

This connection also transfers the convergence theory of FW+ to CCCP+.

\begin{corollary}
\label{cor:cccp+-rate}
Suppose that the sequence $\{x_k\}$ is generated by \eqref{eqn:cccp+} applied to Problem~\eqref{eq:2}. Then, there exists an index $\tau \in [k]$, such that
\begin{equation*}
\begin{gathered}
    f_0(x_\tau) - f_0(x) - \ip{\gr g_0(x_\tau)}{x_\tau-x} \leq \frac{1}{k} \Big((f_0(x_1) - g_0(x_1)) - (f_0^\star - g_0^\star)\Big), \\[0.5em]
    \text{for all $x$ that satisfies} \quad f_i(x) - g_i(x_\tau) - \ip{\gr g_i(x_\tau)}{x - x_\tau} \leq 0, \quad i = 1,2,\ldots,m.
\end{gathered}
\end{equation*}
\end{corollary}

\begin{proof}
By \Cref{fact:cccp+} and \Cref{thm:fw+}, we get
\begin{equation*}
    t_{0,\tau} - t_0 - \ip{\gr g_0(x_\tau)}{x_\tau-x} \leq \frac{1}{k} \Big((t_{0,1} - g_0(x_1)) - (t_0^\star - g_0^\star)\Big) 
\end{equation*}
for all $(x,t_0,\ldots,t_m)$ that satisfies 
\begin{equation*}
\begin{aligned}
    f_i(x) & \leq t_i, & & \quad i = 0,1,\ldots,m \\
    t_i - g_i(x_\tau) - \ip{\gr g_i(x_\tau)}{x - x_\tau} & \leq 0, & & \quad i = 1,2,\ldots,m.
\end{aligned}
\end{equation*}
This is equivalent to the desired bounds by eliminating $t_0,t_1,\ldots,t_m$ in these formula.
\end{proof}

\section{Implications and Discussion}
\label{sec:disc}
\paragraph{Discussion.} We reiterate an important aspect of the non-asymptotic convergence guarantees proved for CCCP and CCCP+ in this paper. Owing to the special structure involving minimization of concave functions, the convergence rates are \textbf{\emph{independent of the curvature}} constant of FW (see~\eqref{eqn:curvature-constant}). This property is valuable since it delivers a convergence analysis of CCCP without imposing usual $L$-smoothness assumptions on either $f$, or $g$ or their difference $f-g$, a limitation that typical first-order methods impose. While surprising at first sight, we can reconcile with this intuitively because CCCP requires solving an optimization subproblem~\eqref{eq:17}, which is a stronger oracle than merely a gradient oracle for $f-g$. The ensuing faster convergence (i.e., tighter bound on the rate) may offer one explanation to the intuitive view on why CCCP can often be quite competitive~\citep{yuille2003concave,lipp2016variations}.

\paragraph{Implications.} Beyond generic convergence guarantees and the transfer of other progress on FW to CCCP, we note some specific implications worth further study here. The EM algorithm is a special case of CCCP~\citep[][\S4]{yuille2003concave}, and therefore a special case of FW. Thus, one should be able to develop a more refined understanding of global iteration complexity of EM, as well as sharper local convergence properties. Given the large number of settings in machine learning and statistics where the EM algorithm, or more generally, variational methods~\citep{blei2017variational} are used, a further deepening of their relation to FW-methods should prove fruitful. Similar implications about algorithms, convergence, and complexity apply to other instances of CCCP that have been studied, notably for problems such as matrix scaling~\citep{sinkhorn1967diagonal}. %
\paragraph{Extensions.} The connection established between FW and CCCP (FW+ and CCCP+) paves the path for numerous extensions that are worthy pursuing. We hope our work motivates others to pursue some of these or other extensions that these connections can inspire.

The first important extension worthy of study is to obtain stochastic variants of CCCP by building on the recent progress on non-convex stochastic FW~\citep{reddi2016stochastic} as well as non-convex variance reduced versions in the case of finite-sum problems~\citep{reddi2016svrg,allen2016variance,yurtsever2019conditional}.

Another line of research worthy of closer attention is to study the case where either $f(x)$ is non-differentiable, or even both $f$ and $g$ in~\eqref{eq:1} are non-differentiable. In that case, developing an extension of the relation between FW and CCCP to a relation between nonsmooth FW~\citep{thekumparampil2020projection} (although, as of now this nonsmooth FW method is limited to convex problems) and DCA should also prove to be fruitful, given the extensive progress DCA has witnessed over the decades~\citep{le2018dc}. The case where $f(x)$ is an indicator function of a convex set already fits trivially, as already discussed above. One might also speculate that progress on DCA could be transferred back to discover and develop nonsmooth FW methods.

Finally, given that both DCA and CCCP are special instances of Majorize-Minorize (MM) methods~\citep{hunter2004tutorial}, it is possible that other MM methods might be profitably viewed as instances of an optimization procedure (not necessarily FW)  whose geometric and convergence properties are better understood. We hope our work inspires a discovery of additional such connections.

\section*{Acknowledgements}
\label{sec:acknowledgements}

Alp Yurtsever received support from the Wallenberg AI, Autonomous Systems and Software Program (WASP) funded by the Knut and Alice Wallenberg Foundation. Suvrit Sra acknowledges support from the NSF CAREER grant (1846088) and the NSF CCF-2112665 (TILOS AI Research Institute). 

\bibliographystyle{plainnat}
\setlength{\bibsep}{3pt}
\bibliography{ref} 

\newpage
\appendix
\begin{center}
\Large\bf Appendix: Additional Implications and Connections    
\end{center}
In this appendix we briefly comment on a few additional connections and implications. The main text is self-contained and can be read independent of this appendix.

\section{Proximal Point Method and Frank-Wolfe}
Consider the generic unconstrained optimization template
\begin{equation}
  \label{eqn:generic-optim}
    \min_x\quad f(x).
\end{equation}
The first connection worth noting is related to the fundamental \emph{proximal point method} (PPM)~\citep{martinet1970regularisation,rockafellar1976monotone}:
\begin{equation}
    x_{k+1} \gets \argmin_x \bigl(f(x)+\tfrac{1}{2\lambda_k}\|x-x_k\|^2\bigr).
\end{equation}
Rewrite \eqref{eqn:generic-optim} by adding and subtracting a strongly convex function $\phi(x)$:
\begin{equation}
\label{eqn:generic-optim-reform}
    \min_x \quad f(x) + \phi(x) -\phi(x).
\end{equation}
We can apply the idea of epigraph splitting:
\begin{equation}
\label{eqn:generic-optim-epi}
    \min_{x,t} \quad t -\phi(x)  \quad \subjto \quad f(x) + \phi(x)\leq t.
\end{equation}
This is concave minimization over a convex set. 
Similar to the ideas present in \Cref{sec:basic-cccp}, application of FW to \eqref{eqn:generic-optim-epi} leads to 
\begin{equation*}
  x_{k+1} \gets \argmin_x\quad f(x)+D_\phi(x,x_k),\quad k=0,1,\ldots
\end{equation*}
which is exactly the Bregman-PPM update. 
If in particular $\phi(x)=\half\|x\|^2$, with a suitable step-size choice in FW, we immediately obtain the PPM iteration in~\eqref{eqn:generic-optim}. Notably, in the above discussion $f$ need not be smooth, and we have indeed no loss in generality due to the FW view. Since many other methods in optimization reduce to PPM, and above we have reduced PPM to FW, many other optimization methods also reduce to FW via this view. Exploring implications of such algorithmic reductions is left as a topic for future study.

\section{Mirror Descent and Frank-Wolfe}
We revisit \eqref{eqn:generic-optim-reform} but this time consider the following epigraph form:
\begin{equation}
\label{eqn:mirror-descent-template}
    \min_{x,t} \quad t + f(x) - \phi(x) \quad \subjto \quad \phi(x) \leq t.
\end{equation}
The objective is not concave in this case, but we can still use FW with a suitable step-size if the curvature bounded. The linear minimization step becomes (after eliminating $t$):
\begin{equation*}
    x_k^* = \argmin_{x} \quad \phi(x) + \ip{\gr f(x_k) - u_k}{x}, 
\end{equation*}
where $u_k \in \partial \phi (x_k)$ is a subgradient of $\phi$ at $x_k$. This leads to the inclusion %
\begin{equation*}
    u_k - \gr f(x_k) \in \partial \phi(x_k^*). 
\end{equation*}
From the convex combination step of FW, we have the relation
\begin{equation*}
    x_k^* = \frac{1}{\eta_k} x_{k+1} + (1 - \frac{1}{\eta_k}) x_k.
\end{equation*}
By combining the last two relations, and rearranging, we get
\begin{equation}
    x_{k+1} = \gr \phi^* \big(u_k - \eta_k \gr f(x_k) \big), \quad u_k \in \partial \phi(x_k),
\end{equation}
where $\phi^*(y)$ is the Fenchel conjugate of $\phi$, and $\phi^*$ is smooth since $\phi$ is strongly convex. This is exactly the mirror descent update. 
If in particular $\phi(x)=\frac{L}{2}\|x\|^2$, the objective function in \eqref{eqn:mirror-descent-template} becomes concave and we can choose $\eta_k = 1$, which leads to the gradient descent step
\begin{equation}
    x_{k+1} = x_k - \frac{1}{L} \gr f(x_k).
\end{equation}

\subsection{Mirror Prox and Frank-Wolfe}
This time, we consider a generic composite optimization problem,
\begin{equation}
\label{eqn:composite-template}
    \min_{x} \quad f(x) + g(x).
\end{equation}
Suppose $f$ is $L$-smooth. Then, $f(x) - \frac{L}{2}\norm{x}^2$ is concave. 
Add and subtract $\frac{L}{2}\norm{x}^2$ in \eqref{eqn:composite-template}:
\begin{equation}
\label{eqn:composite-template-reform}
    \min_{x} \quad f(x) - \frac{L}{2}\norm{x}^2 + g(x) + \frac{L}{2}\norm{x}^2,
\end{equation}
and consider the following epigraph form
\begin{equation}
\label{eqn:composite-template-epigraph}
    \min_{x,t} \quad f(x) - \frac{L}{2}\norm{x}^2 + t \quad \subjto \quad g(x) + \frac{L}{2}\norm{x}^2 \leq t.
\end{equation}
If we apply FW on this formulation, the linear minimization step becomes
\begin{equation}
    x_k^* = \argmin_x \quad \ip{\gr f(x_k) - L \cdot x_k}{x} + g(x) + \frac{L}{2}\norm{x}^2,
\end{equation}
which is exactly the proximal gradient method step
\begin{equation}
    x_k^* = \prox_{\frac{1}{L}g} \big(x_k - \frac{1}{L}\gr f(x_k)\big). 
\end{equation}
Since the objective in \eqref{eqn:composite-template-epigraph} is concave, we can choose unit step-size for FW and get $x_k^* = x_{k+1}$. 

This idea trivially extends to the mirror-prox by using a Bregman divergence. 

\section{CCCP/FW and Non-Convex Proximal-Splitting}
Consider again the basic unconstrained problem $\min_x f(x)-g(x)$. Add and subtract $\half\|x\|^2$, and then consider the equivalent \emph{dual CCCP problem}
\begin{equation}
  \label{eq:supp16}
  (g+\half\|\cdot\|^2)^*(x) - (f + \half\|\cdot\|^2)^*(x).
\end{equation}
When applying CCCP to~\eqref{eq:supp16} we need to compute the gradient of the second term. This gradient can be obtained by recognizing that $f+\half\|x\|^2$ is strongly convex, and hence its dual is smooth, so that
\begin{equation*}
  \gr (f + \half\|\cdot\|^2)^*(z) = \argmax_{x}\ip{z}{x}-f(x)-\half\norm{x}^2.
\end{equation*}
But the argmax above is nothing but $\prox_f(z)$. Thus, to apply CCCP to~\eqref{eq:supp16} we need to solve the subproblem
\begin{equation}
  \label{eq:supp17}
  x_k^* \gets \argmin_x\ (g+\half\|\cdot\|^2)^*(x) - \ip{\prox_f(x_k)}{x}.
\end{equation}
To solve~\eqref{eq:supp17}, we can again compute the dual. As a shorthand, write $h^*\equiv (g+\half\norm{\cdot}^2)^*$. Thus, we obtain that $x_{k}^* = \argmax \ip{\prox_f(x_k)}{x}-h^*(x)$, whereby it must satisfy the optimality condition
\begin{equation*}
    \prox_f(x_k)=\gr h^*(x_k^*).
\end{equation*} 
From our previous argument (since $h=g+\half\norm{\cdot}^2$), we know that $\nabla h^*(x)= \argmax_z\ip{x}{z}-g(z)-\half\norm{z}^2$, so that
\begin{equation}
    \label{eq:xyz}
    \gr h^*(x_k^*) = \prox_g(x_k^*) = \prox_f(x_k).
\end{equation}
We state this relation only to highlight it as the prox-analog of the usual implicit update of CCCP, namely, $\nabla f(x_k^*)=\nabla g(x_k)$.

After computing $x_k^*$ via~\eqref{eq:supp17}, or via~\eqref{eq:xyz}, the update step of FW leads to
\begin{equation}
  \label{eq:18}
  x_{k+1} = (1-\eta_k)x_k + \eta_k x_k^*,
\end{equation}
Ultimately, the dual view permits us to apply CCCP without assuming differentiability, while obtaining convergence guarantees for it via the FW view applied to the dual problem~\eqref{eq:supp17}.

\section{Frank-Wolfe is a special case of CCCP}

We showed in \Cref{sec:basic-cccp} that CCCP is special instance of FW. Interestingly (though unsurprisingly), this connection goes both ways and we can cast FW as an instance of CCCP.

\begin{example}
  \label{ex:fw}
  Let $f(x)$ be an indicator function of a convex set $\Xc$. Then, 
  \begin{equation*}
    \min_x\ f(x)-g(x)\quad\Leftrightarrow\quad \max g(x),\ \text{s.t.}\ x \in \Xc.
  \end{equation*}
  Applying FW to this problem results in the same update, when solving the CCCP subproblem
  \begin{equation*}
    \min_x\ f(x) - g(y)-\ip{\gr g(y)}{x-y}.
  \end{equation*}
  Since $f$ is an indicator, this problem translates into the linear minimization step of FW
  \begin{equation*}
    x_k^* \in \argmin_x\ \ip{-\gr g(x_k)}{x},\quad\text{s.t.}\ x \in \Xc.
  \end{equation*}
\end{example}
Since the objective is concave, FW can use the unit step-size and $x_{k+1} = x_k^*$.  

Example~\ref{ex:fw} shows that FW is a special case of CCCP. But given the previous discussion, we saw that CCCP may be realized via FW too. So, in the end, are they formally equivalent?

\section{More details on FW+}
\label{sec:supp-fw+}

In \Cref{sec:FW+} we analyzed the FW+ method for the important setting where both $\phi$ and $\psi$ are concave. In this appendix, we present additional details on the convergence guarantees of FW+ when $\phi$ and $\psi$ are convex.

\begin{theorem}
Suppose that the sequence $\{\omega_k\}$ is generated by the FW+ algorithm with step-size $\eta_k = 2/(k+1)$ applied to Problem~\eqref{eqn:FW+-template}. Suppose $\phi$ and $\psi$ are convex with bounded curvature constants $C_\phi$ and $C_\psi$. Then, the following guarantees hold:
\begin{equation}
    \omega_k \in \Dc, \quad \phi(\omega_k) - \phi^* \leq \frac{2 C_\phi}{k+1}, \quad \text{and} \quad \psi(\omega_k) \leq \frac{2 C_\psi}{k+1}.
\end{equation}
\end{theorem}

\begin{proof}
$\omega_k$ always remains in $\Dc$ by design of the algorithm. 
Next, focus on the constraint function $\psi$. Since $\psi$ has bounded curvature $C_\psi$, 
\begin{equation*}
\begin{aligned}
    \psi(\omega_{k+1}) 
    & \leq \psi(\omega_k) + \ip{\gr \psi(\omega_k)}{\omega_{k+1} - \omega_k} + \frac{1}{2}\eta_k^2 C_\psi \\
    & = (1-\eta_k) \psi(\omega_k) + \eta_k (\psi(\omega_k) + \ip{\gr \psi(\omega_k)}{\omega^*_k - \omega_k}) + \frac{1}{2}\eta_k^2 C_\psi \\
    & \leq (1-\eta_k) \psi(\omega_k) + \frac{1}{2}\eta_k^2 C_\psi,
\end{aligned}
\end{equation*}
where the last inequality follows from the constraint in the linear minimization step of FW+. Unrolling this inequality, we obtain
\begin{equation*}
    \psi(\omega_{k+1}) \leq (1-\eta_1)\psi(\omega_1) + \frac{1}{2} C_\psi \sum_{i=1}^k \eta_i^2 \prod_{j=i+1}^k (1-\eta_j).
\end{equation*}
We choose $\eta_k = 2/(k+1)$, hence 
\begin{equation*}
    \psi(\omega_{k+1}) 
    \leq 2 C_\psi \sum_{i=1}^k \frac{1}{(i+1)^2} \prod_{j=i+1}^k \frac{j-1}{j+1}
    = 2 C_\psi \sum_{i=1}^k \frac{1}{(i+1)^2} \frac{i(i+1)}{k(k+1)}
    \leq \frac{2 C_\psi}{k+1}.
\end{equation*}
This completes the proof for the constraint function. 

The proof of convergence for the objective function follows similarly to the standard proof of the FW method. By the bounded curvature assumption, we have
\begin{equation*}
\begin{aligned}
    \phi(\omega_{k+1}) - \phi^*
    & \leq \phi(\omega_k) - \phi^* + \ip{\gr \phi(\omega_k)}{\omega_{k+1} - \omega_k} + \frac{1}{2}\eta_k^2 C_\phi \\
    & = \phi(\omega_k) - \phi^* + \eta_k \ip{\gr \phi(\omega_k)}{\omega_k^* - \omega_k} + \frac{1}{2}\eta_k^2 C_\phi.
\end{aligned}
\end{equation*}
By definition of $w_k^*$,
\begin{equation*}
\begin{multlined}
    \phi(\omega_{k+1}) - \phi^*
    \leq \phi(\omega_k) - \phi^* + \eta_k \ip{\gr \phi(\omega_k)}{\omega - \omega_k} + \frac{1}{2}\eta_k^2 C_\phi, \qquad \\ 
    \qquad \forall \omega \in \{\omega \in \Dc: \psi(\omega_\tau) + \ip{\gr \psi(\omega_\tau)}{\omega - \omega_\tau} \leq 0 \}.
\end{multlined}
\end{equation*}
Since $\psi$ is convex, this is a relaxation of the feasible set, it contains $\omega^*$, hence the above inequality holds also for $\omega = \omega^*$: 
\begin{equation*}
\begin{aligned}
    \phi(\omega_{k+1}) - \phi^*
    & \leq \phi(\omega_k)- \phi^* + \eta_k \ip{\gr \phi(\omega_k)}{\omega^* - \omega_k} + \frac{1}{2}\eta_k^2 C_\phi \\
    & \leq (1-\eta_k) (\phi(\omega_k)- \phi^*) + \frac{1}{2}\eta_k^2 C_\phi.
\end{aligned}
\end{equation*}
Then, similar to the above discussion on the convergence of $\psi$, we get
\begin{equation*}
\begin{aligned}
    \phi(\omega_{k+1}) - \phi^*
    & \leq \frac{2 C_\phi}{k+1}.
\end{aligned}
\end{equation*}
This completes the proof.
\end{proof}

\end{document}